\documentclass[12pt]{amsart}
\usepackage{amsmath, amssymb,graphicx,mathrsfs,enumerate, verbatim}
\usepackage{hyperref}% Links in text
\usepackage{placeins}% Floatbarrier

\newtheorem{thm}{Theorem}
\newtheorem{prop}[thm]{Proposition}
\newtheorem{lem}[thm]{Lemma}
\newtheorem{cor}[thm]{Corollary}

\theoremstyle{definition}

\newcommand{\GG}{{\mathbb G}}

\newcommand{\ZZ}{{\mathbb Z}}

\newcommand{\CC}{{\mathbb C}}

\newcommand{\SL}{\operatorname{SL}}

\newcommand{\Sp}{\operatorname{Sp}}

\newcommand{\Hom}{\operatorname{Hom}}

\newcommand{\Lie}{\operatorname{Lie}}

\begin{document}

\title[Minimal epimorphic subgroups]{On minimal epimorphic subgroups in simple algebraic groups of rank $2$}
\author{I.I. Simion}
\address[I.I. Simion]
        { Department of Mathematics\\
          Babe{\c s}-Bolyai University\\
          Str. Ploie\c sti 23-25, 400157, Cluj-Napoca, Rom\^ania}
        \email{iulian.simion@ubbcluj.ro}
\author{D.M. Testerman}
\address[D.M. Testerman]
        { Institute of Mathematics\\
          \'Ecole Polytechnique F\'ed\'erale de Lausanne,
          Station 8, Lausanne, CH-1015, Switzerland
        }
        \email{donna.testerman@epfl.ch}
\thanks{Both authors acknowledge the support of the Swiss National Science Foundation through grant number ${\rm{IZSEZ0}}\textunderscore190091$. I.I. S. was also supported by a grant of the Romanian Ministry of Research, Innovation and Digitalization, CNCS/CCCDI–UEFISCDI, project number PN-III-P4-ID-PCE-2020-0454, within PNCDI III}
%random walk, finite simple group}
\maketitle
%\tableofcontents

\begin{abstract}
  The category of linear algebraic groups admits non-surjective epimorphisms. For simple algebraic groups of rank $2$ defined over algebraically closed fields, we show that the minimal dimension of a closed epimorphic subgroup is $3$.
\end{abstract}

\section{Introduction}

Let $\phi:H\rightarrow G$ be a homomorphism of linear algebraic groups defined over an algebraically closed field. The map $\phi$ is an epimorphism if it admits right cancellation, i.e. whenever $\psi_1\circ\phi=\psi_2\circ\phi$ for homomorphisms $\psi_1,\psi_2$ we have $\psi_1=\psi_2$. %Equivalently, the inclusion map $\phi(H)\rightarrow G$ is an epimorphism.
An \emph{epimorphic subgroup} $H\subseteq G$ is a subgroup for which the inclusion map is an epimorphism. If $\phi$ is a non-surjective epimorphism then $\phi(H)$ is a proper epimorphic subgroup of $G$.

A study of epimorphic subgroups in linear algebraic groups was initiated by Bien and Borel in \cite{BB1,BB2} where criteria for recognizing epimorphic subgroups are established (see \S\ref{subsec:criteria}). Their work follows \cite{Bergman}, where non-surjectivity of epimorphisms is studied for Lie algebras, and extends the list of categories for which the difference between surjectivity and epimorphicity was studied in \cite{Reid}. It follows from \cite{BB1} that minimal epimorphic subgroups are solvable and that maximal proper epimorphic subgroups are parabolic subgroups. Presently, a classification of epimorphic subgroups appears out of reach. 

While maximal epimorphic subgroups are well understood, the current state of knowledge on minimal dimensional epimorphic subgroups is less satisfactory. In case $G$ is a simple algebraic group, the minimal dimension of such subgroups appears to be $3$. If the ground field is $\CC$, this is shown to hold by Bien and Borel in \cite{BB1} (see \S\ref{subsec:criteria}). Here we consider simple groups of rank $2$ defined over fields of positive characteristic.
\begin{thm}
  \label{thm:epi}
  For a simple algebraic group of rank $2$ defined over an algebraically closed field of characteristic $p>0$, the minimal dimension of a closed epimorphic subgroup is $3$.
\end{thm}
In general, for arbitrary rank, we expect that there are no $2$-dimensional closed epimorphic subgroups in simple algebraic groups. The existence of $3$-dimensional such subgroups is studied in \cite{DonnaAdam}.

\begin{comment}One characterization of epimorphic subgroups $H$ of $G$ says that on any rational $G$-module the fixed points of $G$ coincide with the fixed points of $H$ (see \S\ref{subsec:criteria}), in other words $H^{0}(G,V)=H^{0}(H,V)$ (see \cite[I\S4]{Jantzen_Reductive}) for any rational $G$-module. Moreover, if $H$ is not epimorphic, Lemma \ref{lem:induced} shows that there is an induced module $V$ on which the fixed points don't coincide. Thus we have the following direct consequence of Theorem~\ref{thm:epi}.
%Here $\Mod(G)$ stands for the category of rational $G$-modules.

\begin{cor}
  Let $G$ be a simple algebraic group of rank $2$ over an algebraically closed field of characteristic $p>0$. If $H$ is a closed subgroup of $G$ of dimension at most $2$ then there exists a nonzero dominant weight $\lambda$ for which
  $$
  H^{0}(H,H^{0}(\lambda))\neq 0.
  $$
%  the induced module $H^0(\lambda)$ has a nonzero fixed point for $H$.
%  an induced module $H$
%  Let $G$ be a simple algebraic group of rank $2$ over an algebraically closed field of characteristic $p>0$. If $H$ is a closed subgroup of $G$ of dimension at most $2$ then the subcategory $\Mod(G)$ of $\Mod(H)$ is not invariant under the fixed point functor $H^{0}(H,\_)=(\_)^{H}$.
%there exists a rational $G$-module $V$ such that $$\dim H^{0}(G,V)\lneq \dim H^{0}(H,V).$$
  \end{cor}
\end{comment}

For group schemes of finite type over a field, Brion \cite{Brion2017} gives a characterization of epimorphic subgroups and generalizes some of the criteria in \cite{BB1} for subgroups to be epimorphic.

In the context of Hopf algebras the analogous problem is that of non-injective monomorphisms. Non-surjective epimorphisms and non-injective monomorphisms of Hopf algebras are studied in \cite{Chirvasitu} and \cite{Agore}.

%Another way of describing epimorphic subgroups $H$ of $G$ is via the property that all finite-dimensional irreducible $G$-modules are indecomposable $H$-modules. Considering this description in the category of complex Lie algebras, one obtains the notion of wide subalgebra which is studied in \cite{Panyushev2014}. 

\bigskip

In \S\ref{sec:preliminaries} we recall known facts on epimorphic subgroups and fix the notation needed in the rest of the paper. In particular, in Corollary~\ref{cor:up_to_isogeny} we show that one may restrict to the consideration of simply connected groups $G$. Then, in Proposition~\ref{prop:non-existence}, for the groups $\SL_3(k)$, $\Sp_4(k)$ and $G_2(k)$ we establish the nonexistence of closed connected epimorphic subgroups of dimension at most $2$. For the existence statement, we exhibit such groups of dimension $3$ in Section~\ref{sec:existence}. Our calulations use the structure constants available in \textsc{GAP} \cite{GAP}. All calculations were checked both by hand and with \textsc{GAP}.

\section{Preliminaries}
\label{sec:preliminaries}

Throughout, $G$ denotes a linear algebraic group over an algebraically closed field $k$ of characteristic $p\geq 0$.

\subsection{Criteria for epimorphic subgroups}
\label{subsec:criteria}
For a closed subgroup $H$ of $G$, the algebra $k[G]^{H}$ of regular functions on $G$ which are invariant under the right action of $H$ on $G$ identifies with the algebra $k[G/H]$ of regular functions on $G/H$. Bien and Borel established the following theorem.

\begin{thm}[{\cite[Th\'eor\`eme 1]{BB1}}]
  \label{thm:epi_characterization}
  Let $G$ be a connected linear algebraic group and let $H$ be a closed subgroup of $G$. The following conditions are equivalent:
  \begin{enumerate}[{\rm (1)}]
  \item\label{item:BB1} The subgroup  $H$ is epimorphic in $G$.
  \item\label{item:BB2} We have $k[G/H]=k$.
  \item\label{item:BB3} The $k$-vector space $k[G/H]$ has finite dimension.
  \item\label{item:BB4} The fixed points of $H$ in any rational $G$-module coincide with the fixed points of $G$.
  \item\label{item:BB5} Any direct sum decomposition of a rational $G$-module $V$ as $H$-module is a direct sum decomposition of $V$ as $G$-module.
    \end{enumerate}
\end{thm}

As a consequence of the above theorem and of the definition, for a subgroup $H$ of $G$ we have the following criteria for epimorphicity (see \cite[\S2]{BB1}).

\begin{enumerate}[{\rm (1)}]
  \setcounter{enumi}{5}
  \item\label{item:BB6} For subgroups $H\subseteq L\subseteq G$, if $H$ is epimorphic in $L$ and $L$ is epimorphic in $G$ then $H$ is epimorphic in $G$.
  \item\label{item:BB7} The subgroup $H$ is epimorphic in $G$ if and only if the connected component of the identity $H^{\circ}$ or a Borel subgroup of $H^{\circ}$ is epimorphic in $G$.
  \item\label{item:BB8} If $G$ is simple, a proper subgroup $H$ is epimorphic if and only if the radical of $H^{\circ}$ is epimorphic in $G$.
  \item\label{item:BB9} If $(G_{i})_{i\in I}$ is a family of closed subgroups which generate $G$ and if $H\cap G_{i}$ is epimorphic in $G_i$ for each $i\in I$ then $H$ is epimorphic in $G$.
\end{enumerate}

The characterization in Theorem~\ref{thm:epi_characterization} shows that if $H$ is epimorphic in $G$ then $G/H$ is close to being projective. At the other extreme, the following result, established independently by Richardson and by Cline, Parshall and Scott, provides a criterion for determining when $G/H$ is affine.

\begin{thm}[{\cite[Theorem A]{Richardson77} and \cite[Corollary 4.5]{CPS77}}]
  \label{thm:aff_characterization}
  Let $H$ be a closed subgroup of the reductive group $G$.
  Then $G/H$ is an affine variety if and only if $H$ is reductive.
\end{thm}

\begin{cor}
  \label{cor:reductive_proper}
  Let $H$ be a closed subgroup of the reductive group $G$.
  If $H$ is reductive and $\dim(H)\neq\dim (G)$ then
  no closed subgroup of $H$ is epimorphic in $G$.
\end{cor}
\begin{proof}
  By Theorem~\ref{thm:aff_characterization} and the assumption that $\dim(H)\neq\dim (G)$ the dimension of $k[G/H]$ is not finite. Thus $H$ is not epimorphic in $G$ by Theorem~\ref{thm:epi_characterization}.(3). If $M$ is a closed subgroup of $H$ then
  $$
  k[G/M]\cong k[G]^{M}\supseteq k[G]^{H}\cong k[G/H],
  $$
  hence $M$ cannot be epimorphic in $G$.
%  This follows from Theorem~\ref{thm:aff_characterization}. However, we can also give a short direct argument. Let $H$ be a subgroup of $G$ and suppose $H\subseteq M\subsetneq G$ for a reductive group $M$ with Lie algebra $\Lie(M)$ of dimension $m$. Then the line $\bigwedge\nolimits^m\Lie(M)$ in $\bigwedge\nolimits^m\Lie(G)$ is fixed by $M$ and therefore by $H$ but not by $G$.
  \end{proof}

The following consequence of Theorem~\ref{thm:aff_characterization} was observed by Bien and Borel in {\cite{BB1}. The proof follows from the fact that for $k=\CC$ any two-dimensional, non-abelian, non-unipotent closed connected subgroup of a simple algebraic group over $k$, lies in an $A_1$-type subgroup.

\begin{cor}
  \label{cor:char_0}
  There are no two-dimensional closed epimorphic subgroups in a simple algebraic group of rank at least $2$ defined over $\CC$.
\end{cor}

In the rest of the paper we assume that the characteristic of the ground field is positive, i.e. $p>0$.

\begin{comment}The proof of the following Lemma was pointed out to us by George McNinch.

\begin{lem}\label{lem:induced}
  Let $G$ be a connected reductive algebraic group and let $H\subseteq G$ be a nonepimorphic subgroup of $G$. Then
  there exists a nonzero dominant weight $\lambda$ for which the induced module $H^0(\lambda)$ has a nonzero fixed point for $H$.

\end{lem}

\begin{proof} We use the fact, \cite[II, 4.20]{Jantzen_Reductive}, that $k[G]$ has a good filtration, and that $H^0(0)$ occurs with multiplicity $1$ in this filtration. Now $H^0(0) = k$ is a submodule of $k[G]$ and has a good filtration as well. Thus by \cite[II, 4.17]{Jantzen_Reductive}, $k[G]/k$ also has a good filtration, none of whose factors are trivial. Let $0\subsetneq E_1\subsetneq E_2\cdots\subsetneq k[G]$ be such a good filtration with $E_1 = H^0(0)$, and for $i\geq 1$, $E_{i+1}/E_i\cong H^0(\lambda_i)$ for some dominant weight $\lambda_i\ne 0$.

  Now since $H$ is not epimorphic in $G$, there exists $f\in k[G]^H$, $f$ nonconstant. Choose $i\geq 2$ minimal such that $f\in E_i$.
 Then $f+E_{i-1}$ is a nonzero fixed point for $H$ in $E_i/E_{i-1}\cong H^0(\lambda_i)$, proving the claim. %giving the result.
\end{proof}

\end{comment}

\subsection{Isogenies of $G$}

The proof of part (2) of the following lemma was communicated to us by Michel Brion.

\begin{prop}
  \label{lem:up_to_bij_hom}
  Let $\phi:G_1\rightarrow G_2$ be a surjective homomorphism of algebraic groups. %Then
  \begin{enumerate}[{\rm (1)}]
  \item If $H_1$ is a closed epimorphic subgroup of $G_1$ then $\phi(H_1)$ is epimorphic in $G_2$.
  \item If $H_2$ is a closed epimorphic subgroup of $G_2$ then $\phi^{-1}(H_2)$ is epimorphic in $G_1$.
  \end{enumerate}
\end{prop}
\begin{proof}
  Let $H_2$ denote the image $\phi(H_1)$ of the closed epimorphic subgroup $H_1$ of $G$. Any rational $G_2$-module $V$ is a rational $G_1$-module via $\phi$, thus $V^{H_2}=V^{H_1}=V^{G_1}=V^{G_2}$. Hence $H_2$ is epimorphic in $G_2$.
  
  Suppose now that $H_2$ is a closed epimorphic subgroup in $G_2$ and denote $\phi^{-1}(H_2)$ by $H_1$. The map $\phi$ induces a bijective morphism
  of varieties $G_1/H_1\rightarrow G_2/H_2$. Thus, the function field $k(G_1/H_1)$ is a purely inseparable finite extension of $k(G_2/H_2)$. Hence, since $p>0$, there exists a $p$-power $q$ such that $k(G_1/H_1)^q$ lies in $k(G_2/H_2)$. In particular $k[G_1/H_1]^q$ lies in $k(G_2/H_2)$. Since $G_2/H_2$ is a normal variety (see \cite[Lemma 5.3.4]{Springer}), it follows that $k[G_1/H_1]^q$ lies in $k[G_2/H_2]=k$. Hence, $k[G_1/H_1]=k$.
\end{proof}

In what follows, we refer to the isomorphisms of abstract groups described in \cite[Theorem 28]{Steinberg} as \emph{exceptional isogenies}. These are non-separable bijective endomorphisms of simple algebraic groups and occur for groups of type $B_2$ and $F_4$  when $p=2$, and for groups of type $G_2$ when $p=3$.

%In what follows, we will adopt the terminology ``exceptional isogeny'', as used by Springer in \cite[\S9.6]{Springer},
%for non-separable bijective endomorphisms of simple algebraic groups, namely those occurring for groups of type $B_2$ and $F_4$  when $p=2$, and for groups of type 
%$G_2$ when $p=3$. In \cite[\S22]{Borel_book}, such morphisms are called ``quasi-central''. A precise description of these can be found in \cite[\S11]{Steinberg}.

\begin{cor}
  \label{cor:up_to_isogeny}
  Let $G$ be a simple algebraic group and let $H$ be a closed subgroup of $G$.
  \begin{enumerate}[{\rm (1)}]
  \item For any isogeny $\phi : G \to G'$, the subgroup $H$ is epimorphic in G if and only if $\phi(H)$ is epimorphic in $G'$.

  \item The dimension of a minimal closed epimorphic subgroup of $G$ is independent of the isogeny type of $G$.
    \end{enumerate}
\end{cor}

\subsection{Root system}
\label{subsec:root_system}
Throughout we fix a maximal torus $T$ and a Borel subgroup $B$ of $G$, with $T\subseteq B$, and denote by $U$ the unipotent radical of $B$. The root system $\Phi$ is with respect to $T$ and the positive roots and the simple roots are with respect to $U$. For $\SL_3(k)$ the positive roots are $\alpha_1,\alpha_2,\alpha_3=\alpha_1+\alpha_2$, for $\Sp_4(k)$ the positive roots are $\alpha_1,\alpha_2,\alpha_3=\alpha_1+\alpha_2,\alpha_4=\alpha_1+2\alpha_2$ and for $G_2(k)$ they are
$$
\alpha_1,
\alpha_2,
\alpha_3=\alpha_2+\alpha_1,
\alpha_4=\alpha_2+2\alpha_1,
\alpha_5=\alpha_2+3\alpha_1,
\alpha_6=2\alpha_2+3\alpha_1.
$$
The set of coroots is denoted by $\Phi^{\vee}$ and $\alpha_{i}^{\vee}:\GG_m(k)\rightarrow T$ is the simple coroot corresponding to the simple root $\alpha_i$. Furthermore, we write $\omega_i$ for the fundamental dominant weight corresponding to the simple root $\alpha_i$, for $i=1,2$, and $s_\alpha$ for the reflection in the Weyl group $N_G(T)/T$ associated to the root $\alpha$. In addition, for groups of type $A_1$, we will identify the set of dominant weights with the set of nonnegative integers.

\section{Nonexistence of $2$-dimensional closed epimorphic subgroups}
\label{sec:non-existence}

In this section we show that there are no epimorphic subgroups of dimension at most $2$ in simple algebraic groups of rank $2$. By Corollary~\ref{cor:up_to_isogeny}, it suffices to consider simply connected groups. Thus we treat $\SL_3(k)$, $\Sp_4(k)$ and $G_2(k)$.  Moreover, by \S\ref{subsec:criteria}.\eqref{item:BB7}, it suffices to show the nonexistence of closed connected epimorphic subgroups. If the dimension of such a subgroup is at most $2$, it is necessarily solvable. For a closed connected solvable subgroup $H$, we may assume that $H\subseteq B$. Then, the unipotent radical $U_H$ of $H$ lies in $U$ and we may fix a maximal torus $T_H$ of $H$ lying in $T$. Clearly, $H$ is the semidirect product $U_H\rtimes T_H$. % \cite[IV,\S2 Proposition 3.5]{Demazure_Gabriel}.

Some cases can easily be seen not to appear as epimorphic subgroups.

\begin{lem}
  \label{lem:mixed_2_dim}
  Let $H$ be a closed subgroup of a simple algebraic group $G$ with $\dim H\leq 2$. If all elements in $H$ are semisimple or if all elements in $H$ are unipotent or if $H$ is abelian then $H$ is not epimorphic in $G$.
\end{lem}
\begin{proof}
  By \S\ref{subsec:criteria}.(7), we may assume that $H$ is connected. If $H$ is abelian then either $H$ is a torus, or $H=H_u$ is unipotent or $H$ is $2$-dimensional and $H=H_sH_u$ with $H_s$ a torus and $H_u\cong \GG_a(k)$. In the latter two cases we notice that $H$ fixes $\Lie(H_u)$ but $G$ doesn't, so by Theorem~\ref{thm:epi_characterization}.(4) $H$ is not epimorphic in $G$. If $H$ is a  torus then $H$ acts completely reducibly on every $G$-module, so by
  Theorem~\ref{thm:epi_characterization}.(5) is not epimorphic in $G$.   
\end{proof}

By Lemma~\ref{lem:mixed_2_dim}, for closed subgroups $H$ of simple algebraic groups $G$, it suffices to treat the cases where $H=U_H\rtimes T_H$ is a proper, i.e. non-abelian, semidirect product with $\dim(U_H)=\dim(T_H)=1$. The arguments to follow make use of surjective homomorphisms
$$
t:k^{\times}\rightarrow T_H\subseteq T
\quad\text{and}\quad
u:k\rightarrow U_H\subseteq U.
$$
Here $t\in \Hom(\GG_m(k),T)$ is a cocharacter \cite[II,Ch1,\S1.6]{Jantzen_Reductive}. Thus, if $G$ is simply connected of rank $r$, then $\Hom(\GG_m(k),T)=\ZZ\Phi^{\vee}$ and
$$
t=m_1\alpha_1^{\vee}+m_2\alpha_2^{\vee}+\dots+m_r\alpha_r^{\vee},
$$
a linear combination of simple coroots with $m_1,\dots,m_r\in\ZZ$, i.e. for any $\lambda\in k^{\times}$
$$
t(\lambda)=\alpha_1^{\vee}(\lambda^{m_1})\alpha_2^{\vee}(\lambda^{m_2})\cdots \alpha_r^{\vee}(\lambda^{m_r}).
$$
Describing the map $u$ onto the subgroup $U_H$ is more delicate. Since the ground field is algebraically closed, any $1$-dimensional closed connected unipotent group is isomorphic to $\GG_{a}(k)$. % \cite[IV,\S2 Corollaire 2.10]{Demazure_Gabriel}.
Let $N=|\Phi^{+}|$. 
We have an isomorphism between the varieties $k^N$ and $U$, defined by
$$
(x_1,\dots,x_{N})\mapsto\prod_{i=1}^{N}u_{i}(x_i)
$$
for a fixed (but arbitrary) ordering of the positive roots and for fixed isomorphisms onto the root subgroups $u_{i}:k\rightarrow U_{\alpha_i}$. Thus, by \cite[Lemma 3.6]{Hartshorne}, for the morphism of varieties $u:k\rightarrow U$ we have
\begin{equation}
\label{u_parametrization}
u(x)=\prod_{i=1}^{N} u_{i}(P_i(x))
\end{equation}
for polynomials $P_i(x)$. Notice that while the maps $t$ and $u$ may not be isomorphisms, we may assume that our subgroup $H$ lies in the image of some $\mu\circ(t\times u)$ where $\mu$ is the multiplication map in $G$. Moreover, imposing the condition that $\mu\circ(t\times u)$ is a homomorphism of groups we obtain an isogeny onto $H$.

Now, since $U_{H}\cong\GG_a(k)$, there is an integer $m\neq 0$ such that for any $\lambda\in k^{\times}$ and $x\in k$ we have
%\begin{equation}
%\label{parameters}
$$
{}^{t(\lambda)}u(x)=u(\lambda^{m} x).
$$
%\end{equation}
Hence, for all $\lambda$ and $x$ we have
$$
%\begin{equation}
% \label{Pi_equations}
\prod_i u_i(P_i(\lambda^{m}x))=u(\lambda^{m}x)={}^{t(\lambda)}u(x)=\prod_i u_i(\lambda^{n_i}P_i(x)),
%\end{equation}
$$
for some integers $n_i$.
Therefore $P_i(\lambda^{m}x)=\lambda^{n_i}P_i(x)$ for each $i$, thus $P_i(x)=c_ix^{q_i}$ and
\begin{equation}
\label{unipotent_poly}
u(x)=\prod_i u_i(c_ix^{q_i})
\end{equation}
for some constants $c_i\in k$ and integers $q_i\geq 0$. In the rest of the paper we focus on groups of rank $2$. Henceforth we assume that
$$G\text{ is simple of rank }2.$$

By the above, the map $u$ is determined by $c_1,\dots,c_{N}\in k$ and by the integers $q_1,\dots,q_{N}\geq 0$, while the map $t$ is determined by $m_1,m_2\in\ZZ$. Imposing the condition that $\mu\circ(t\times u)$ is a homomorphism of groups translates into conditions on these parameters. In Lemma~\ref{lem:structure} we pin down the possible values for these parameters and subsequently, in Proposition~\ref{prop:non-existence}, we use this to show that the corresponding subgroups are not epimorphic. For such an analysis we need the following lemmas.

\begin{lem}
\label{lemma0}
Let $q_1,\dots,q_5$ be nonnegative integers with $q_1$ and $q_2$ powers of $p$. Let $c,c_1,c_2\in k$ with $c_1,c_2\neq 0$ and let $z\geq 1$ be an integer. Consider the polynomial $P=c(a+b)^{z}-ca^{z}-cb^{z}$ in $a$ and $b$.
%If for all $a,b\in k$ the element
%$c(a+b)^{z}-ca^{z}-cb^{z}$ equals
\begin{enumerate}[{\rm (1)}]
\item If $P=c_1a^{q_2}b^{q_1}$
then $z=2q_1=2q_2$, $p\neq 2$ and $c=c_1/2$.
\item If $P=c_1\left(a^{q_1}b^{2q_1}+a^{2q_1}b^{q_1}\right)$
then $z=3q_1$, $p\neq 3$ and $c=c_1/3$.
\item If $P=c_1\left( 2a^{q_1}b^{3q_1}+3a^{2q_1}b^{2q_1}+2a^{3q_1}b^{q_1}\right)$
and $p\neq 2$ then $z=4q_1$ and $c=c_1/2$.
\item If $P=c_1\left(a^{q_1}b^{4q_1}+2a^{2q_1}b^{3q_1}+2a^{3q_1}b^{2q_1}+a^{4q_1}b^{q_1}\right)$
and $p\neq 2$ then $z=5q_1$, $p\neq 5$, $c=c_1/5$.
\item If $P=c_1a^{q_3}b^{2q_1}+c_2a^{q_4}b^{q_1}$
then $z=3q_1=3q_3$, $q_4=2q_1$, $p\neq 3$, and $c=c_1/3=c_2/3$.
\item If $P=c_1a^{q_4}b^{q_1}+c_2a^{q_5}b^{q_2}$ then one of the following holds:
  \begin{enumerate}
  \item[{\rm (I)}]
$z$ is a power of $p$, $q_4=q_5$, $q_1=q_2$ and $c_1+c_2=0$, or

  \item[{\rm (II)}]$z=2q_1$, $q_1=q_2=q_4=q_5$, $p\neq 2$ and $c=(c_1+c_2)/2$, or
    
\item[{\rm (III)}] $q_5=q_1\neq q_2=q_4$.
  
%$z=q_1+q_2$, $q_5=q_1$, $q_4=q_2$, $c_1=c\cdot {z \choose q_1}=c_2$ and $p\mid {q_1+q_2 \choose q_1+i}={q_1+q_2 \choose q_2-i}\Leftrightarrow i\neq 0$.
\end{enumerate}
\end{enumerate}
\end{lem}

\begin{proof}
Notice that in (1)--(5) the constant $c$ is different from zero. 
Let $z=xy$ with $y$ the highest power of $p$ dividing $z$.
Notice also that for (1)--(5) we have $x>1$ since $c_1,c_2\neq 0$.
Since
$$
(a+b)^{z}=(a+b)^{xy}=(a^{y}+b^{y})^{x}=a^{xy}+xa^{y}b^{y(x-1)}+\dots +xa^{y(x-1)}b^{y}+b^{xy}
$$
and since $x$ is prime to $p$, one of the monomials in the expression of $P$ must equal $xa^{y}b^{y(x-1)}$ in each case. Thus, for (1) we have $y(x-1)=y$, so $x=2$ and $q_1=q_2$. Hence $p\neq 2$ and $c=c_1/2$. For (2) we have $a^{y}b^{y(x-1)}=a^{q_1}b^{2q_1}$, hence $y=q_1$ and $x=3$. Therefore $p\neq 3$ and $c=c_1/3$. For (3) we have $a^{y}b^{y(x-1)}=a^{q_1}b^{3q_1}$, hence $y=q_1$, $x=4$ and $c=c_1/2$. For (4) we have $a^{y}b^{y(x-1)}=a^{q_1}b^{4q_1}$, hence $y=q_1$ and $x=5$. Therefore $p\neq 5$ and $c=c_1/5$. For (5) we have $a^{y}b^{y(x-1)},a^{y(x-1)}b^{y}\in \{a^{q_3}b^{2q_1},a^{q_4}b^{q_1}\}$. Hence $y=q_1$, $x=3$, $q_3=q_1$ and $q_4=2q_1$. Therefore $p\neq 3$ and $c=c_1/3=c_2/3$. 

For (6), if $q_1=q_2$, $q_4=q_5$ and $c_1+c_2=0$ then $z$ is a power of $p$. Else $P\neq 0$ and we have $a^{y}b^{y(x-1)},a^{y(x-1)}b^{y}\in \{a^{q_4}b^{q_1},a^{q_5}b^{q_2}\}$. Since $p\mid {x\choose i}\Leftrightarrow p\mid {x\choose x-i}$, we have $q_4=q_2$ and $q_5=q_1$ which are powers of $p$. If $q_1\neq q_2$ we are in Case (III). For $q_1=q_2$ we have
$$
c(a+b)^{z}-ca^{z}-cb^{z}
=
(c_1+c_2)a^{q_1}b^{q_1}.
$$
Then, if $c_1+c_2=0$ we are in Case (I). Else, if $c_1+c_2\neq 0$, by (1), $p\neq 2$, $c=(c_1+c_2)/2$ and $z=2q_1$. This is Case (II).
\end{proof}

\begin{lem}
\label{lemmap}
Let $p$ be a prime and $f\geq 1$ an integer. Then
$$
\frac{2^{f+2}-1}{3},\frac{2^{f+1}+1}{3},\frac{p^f+1}{2},\frac{p^{f+1}+3}{2},\frac{3p^f+1}{2}
$$
are not powers of $2,2,p,p,p$ respectively.
\end{lem}
\begin{proof}
  If for some integer $m$ we have $2^{m}=(2^{f+2}-1)/3$,
  then $m\geq f$ and $2^f(3\cdot 2^{m-f}-4)=-1$ which is not possible. The second case is similar.
  If there is $m$ such that $2p^{m}=p^f+1$
  then  $p\neq 2$, $m\geq f$ and $p^f(2p^{m-f}-1)=1$.
  Thus, $f$ has to equal $0$, a contradiction.
  If there is $m$ such that $2p^{m}=p^{f+1}+3$
  then  $p\neq 2$, $m\geq f+1$ and $p^{f+1}(2p^{m-f-1}-1)=3$ which is not possible.
  If there is $m$ such that $2p^{m}=3p^f+1$
  then  $p\neq 2$, $m\geq f$ and $p^f(2p^{m-f}-3)=1$.
  Thus, $f$ has to equal $0$, a contradiction.
\end{proof}

\begin{lem}
  \label{lem:root_group_case}
  If $U_H$ is a root subgroup of $G$ then $H$ is not epimorphic in $G$.
\end{lem}
\begin{proof}
  If $U_H$ is a root group, $H$ lies in a reductive proper subgroup of $G$, namely a proper Levi subgroup, and the claim follows from Corollary~\ref{cor:reductive_proper}.
  \end{proof}

\begin{lem}
  \label{cr_is_one}
  Let $u$ be as in \eqref{unipotent_poly}. For any two distinct roots $\alpha_i$ and $\alpha_j$, conjugating by an element in $T$ we may assume that $c_i=c_j=1$.
\end{lem}
\begin{proof}
This is a consequence of the fact that two distinct positive roots are linearly independent in the character group $X(T)$, that $\dim(T)\geq 2$ and that the field is algebraically closed.
\end{proof}

We are now in a position to describe the possibilities for a two-dimensional closed connected non-abelian subgroup $H$ (of a rank 2 simply connected simple algebraic group)
which contains both unipotent and semisimple elements. With the discussion at the beginning of the section, we may describe such a group by indicating the tuples $(q_1,q_2,\dots)$, $(c_1,c_2,\dots)$ and $(m_1,m_2)$.

\begin{lem}
\label{lem:structure}
If $U_H$ is not a root group, then the only possibilities for $u:k\rightarrow U_H$ and $t:k^{\times}\rightarrow T_H$, up to conjugacy and exceptional isogenies, are given in Tables $\ref{tab:A2_cases}$, $\ref{tab:B2_cases}$ and $\ref{tab:G2_cases}$ for $G=\SL_3(k)$, $\Sp_4(k)$, $G_2(k)$, respectively, where $m\neq 0$ is an integer, $q_i=p^{f_i}$ for some integers $f_i\geq 0$ and where $c_i\in k$ and $d_i\in k^{\times}$ for $1\leq i\leq 6$.
\end{lem}

\begin{proof}
  Note that since $T_H$ lies in $T$, it suffices to show that $U_H$ is conjugate by an element of $N_G(T)$ to some group in the tables. In particular, we may assume that $U_H$ has a nontrivial projection onto $U_\alpha$, for some simple root $\alpha$.

  \FloatBarrier
\renewcommand{\arraystretch}{1.3}
\begin{table}[ht]
  \begin{tabular}{| c | l | l | l | l |}
    \hline
   Case & $q_i$ & $c_1,c_2,c_3$ & $m_i/m$ & $p$\\
   \hline
  1 &  $q_1,q_1,2q_1$ & $1,1,\frac{1}{2}$ & $q_1,q_1$  & $p\geq 3$\\
    \hline
  2 &  $q_1,-,q_3$ & $1,0,1$ & $\frac{1}{3}(q_1+q_3),\frac{1}{3}(2q_3-q_1)$ & any \\
      \hline
    \end{tabular}
\caption{$H$ in $\SL_3(k)$}
  \label{tab:A2_cases}
\end{table}
\FloatBarrier
  
%  \medbreak
\noindent{\textbf{Cases for $\SL_3(k)$:}} Using the additivity of $u:k\rightarrow U_H$, $u(a)u(b)=u(a+b)$, and identifying the projections onto the root groups, we obtain the following equations
  \begin{equation}  
    \label{eq:add_sys_a2}
    \left\{
    \begin{aligned}
      &c_1(a+b)^{q_1}
      =
      c_1a^{q_1}+c_1b^{q_1}
      \\
      &c_2(a+b)^{q_2}
      =
      c_2a^{q_2}+c_2b^{q_2}
      \\
      &c_3(a+b)^{q_3}
      =
      c_3a^{q_3}+c_3b^{q_3}+c_1c_2a^{q_2}b^{q_1}.
    \end{aligned}
    \right.
  \end{equation}
  \medbreak
  \noindent{\textbf{Case 1:}} If $c_1\neq 0$ and $c_2\neq 0$, by Lemma~\ref{cr_is_one} we may assume that $c_1=c_2=1$. Moreover $q_1$ and $q_2$ are powers of $p$. Thus, from the last equation in \eqref{eq:add_sys_a2} we see that $p\neq 2$, $q_1=q_2=q_3/2$ and $c_3=1/2$. The compatibility of the torus $T_H$ with $U_H$ restricts the possibilities for $t:k^{\times}\rightarrow T_H$ and we get that $m_1=m_2=mq_1$.
  \medbreak
  \noindent{\textbf{Case 2:}} If $c_1\neq 0$ and $c_2=0$, since we exclude root groups, we may assume that $c_3\neq 0$. By Lemma~\ref{cr_is_one}, we may assume that $c_1=c_3=1$. Again, $q_1$ and $q_3$ are powers of $p$ and the compatibility with $T_H$ implies that $m_1=m(q_1+q_3)/3$, $m_2=m(2q_3-q_1)/3$. The case where $c_1=0$ and $c_2\neq0$ is covered by conjugating by an element of $N_G(T)$.

%
%
%
%   B2
%
%
%
%  \medbreak
  \noindent{\textbf{Cases for $\Sp_4(k)$:}}
Using the additivity of $u:k\rightarrow U_H$, $u(a)u(b)=u(a+b)$, and identifying the projections onto the root groups, we obtain the following system:
  \begin{equation}
    \label{eq:add_sys_b2}
    \left\{
    \begin{aligned}
      &c_1(a+b)^{q_1}
      =
      c_1a^{q_1}+c_1b^{q_1}
      \\
      &c_2(a+b)^{q_2}
      =
      c_2a^{q_2}+c_2b^{q_2}
      \\
      &c_3(a+b)^{q_3}
      =
      c_3a^{q_3}+c_3b^{q_3}-c_1c_2a^{q_2}b^{q_1}
      \\
      &c_4(a+b)^{q_4}
      =
      c_4a^{q_4}+c_4b^{q_4}-c_1c_2^2a^{2q_2}b^{q_1}
      %\\
      %&\qquad
      -2c_2^2c_1a^{q_2}b^{q_1+q_2}+2c_2c_3a^{q_3}b^{q_2}.
    \end{aligned}
    \right.
  \end{equation}
%  \medbreak

  \noindent{\textbf{Case 1:}} If $c_1\neq 0$ and $c_2\neq 0$, by Lemma~\ref{cr_is_one} we may assume that $c_1=c_2=1$. Moreover, by the first two equations, $q_1$ and $q_2$ are powers of $p$. Applying Lemma~\ref{lemma0}.(1) we have $c_3\neq 0$, $p\neq 2$, $q_3=2q_1=2q_2$ and $c_3=-1/2$. Thus, the last equation in \eqref{eq:add_sys_b2} is
$$
c_4a^{q_4}+c_4b^{q_4}-2a^{2q_1}b^{q_1}-2a^{q_1}b^{2q_1}=c_4(a+b)^{q_4}.
$$
Since $p\neq 2$, from Lemma~\ref{lemma0}.(2) it follows that $p\neq 3$, $c_4=-2/3$ and $q_4=3q_1$. Imposing the compatibility with $T_H$, we get $m_1=2mq_1$, $m_2=3mq_1/2$.

%\FloatBarrier
\renewcommand{\arraystretch}{1.3}
\begin{table}[ht]
  \begin{tabular}{| c | l | l | l | l |}
    \hline
   Case & $q_i$ & $c_1,c_2,c_3,c_4$ & $m_i/m$ & $p$\\
   \hline
  1 &  $q_1,q_1,2q_1,3q_1$ & $1,1,-\frac{1}{2},-\frac{2}{3}$ & $2q_1,\frac{3}{2}q_1$  & $p\geq 5$\\
    \hline
  2 &  $q_1,-,-,q_4$ & $1,0,0,1$ & $\frac{1}{2}(q_1+q_4),\frac{1}{2}q_4$ & any \\
    \hline
  3 &  $q_1,-,q_1,q_1$ & $1,0,1,d_4$ & $q_1,\frac{1}{2}q_1$  & any \\
    \hline
  4 &  $q_1,-,q_3,-$ & $1,0,1,0$ & $q_3,\frac{1}{2}(2q_3-q_1)$ & any \\
    \hline
  5 & $-,q_2,q_2,2q_2$ & $0,1,1,1$ & $q_2,q_2$ & $p\geq 3$ \\
  \hline
  6 & $-,q_2,q_2,2q_2$ & $0,1,1,d_4$ & $q_2,q_2$ & $p=2$ \\
  \hline
    \end{tabular}
\caption{$H$ in $\Sp_4(k)$}
  \label{tab:B2_cases}
\end{table}
%\FloatBarrier  
%\medbreak
\noindent{\textbf{Case 2:}} Consider System \eqref{eq:add_sys_b2} with $c_1\neq 0$ and $c_2=c_3=0$. Since we exclude root groups, by Lemma~\ref{cr_is_one} we may assume that $c_1=c_4=1$, and we have that $q_1$, $q_4$ are $p$-powers. The compatibility with $T_H$ gives $m_1=m(q_1+q_4)/2$, $m_2=mq_4/2$.
\medbreak
\noindent{\textbf{Case 3:}} If $c_1,c_3,c_4\neq0$ and $c_2=0$ we may assume that $c_1=c_3=1$ by Lemma~\ref{cr_is_one}, and so $q_1$ and $q_3$ are $p$-powers. Since $c_4\neq 0$, the compatibility with $T_H$ yields $m_1=mq_3$, $m_2=mq_4/2$ and $2q_3=q_4+q_1$. The condition $2q_3=q_4+q_1$ implies that $q_1=q_4$, hence $q_1=q_3=q_4$. Indeed if $q_i=p^{f_i}$ and $q_4\geq q_1$, then $2q_3=q_1(p^{f_4-f_1}+1)$ and by Lemma~\ref{lemmap}, this equation doesn't have solutions unless $f_4=f_1$.
\medbreak
\noindent{\textbf{Case 4:}} Consider System \eqref{eq:add_sys_b2} with $c_1,c_3\neq0$ and $c_2,c_4=0$. Again we may assume that $c_1=c_3=1$ by Lemma~\ref{cr_is_one}, so $q_1$ and $q_3$ are again $p$-powers. From the compatibility with $T_H$ we get $m_1=mq_3$, $m_2=m(2q_3-q_1)/2$. Conjugating by an element of $N_G(T)$, we obtain also the case where $c_1=c_3=0$ and $c_2\neq 0$. Indeed, since we exclude root groups, in this case $c_4\neq 0$ and by Lemma~\ref{cr_is_one} we may assume that $c_2=c_4=1$. Conjugating with a representative of $s_{\alpha_2}s_{\alpha_1}$ in $N_G(T)$, we obtain the subgroups treated previously.
\medbreak
\noindent{\textbf{Cases 5,6:}} Let $c_1=0$, $c_2,c_3\neq 0$ and assume first that $p\neq 2$. Then $c_2=c_3=1$ by Lemma~\ref{cr_is_one}, and here we find that $q_2$ and $q_3$ are $p$-powers. By Lemma~\ref{lemma0}.(1) we have $c_4\neq 0$, $q_2=q_3$, $q_4=2q_2$ and $c_4=1$. The compatibility with $T_H$ gives $m_1=m(q_4-q_2)$, $m_2=mq_4/2$.

%\medbreak
%\noindent{\textbf{Case 6:}}
Now let $c_1=0$, $c_2,c_3\neq 0$ and $p=2$. Again, $c_2=c_3=1$ by Lemma~\ref{cr_is_one}. Moreover, System \eqref{eq:add_sys_b2} shows that $q_2$, $q_3$ and $q_4$ are powers of $p$ and $c_4$ is arbitrary. For $c_4=0$, the group $U_H$ is obtained from the group described in Case 2 by applying an exceptional isogeny which interchanges the long and the short root groups. For the case when $c_4\neq 0$, the compatibility with the torus gives $q_4=q_2+q_3$. Thus, since $q_2,q_3,q_4$ are powers of $p=2$, we must have $q_4=2q_2=2q_3$.

%
%
%
%   G2
%
%
%
%\FloatBarrier
\renewcommand{\arraystretch}{1.3}
\begin{table}[ht]
  \resizebox{350pt}{!}{%
  \begin{tabular}{| c | l | l | l | l |}
    \hline
     Case & $q_i$ & $c_1,c_2,c_3,c_4,c_5,c_6$ & $m_i/m$ & $p$\\
    \hline
    1 &  $q_1,q_1,2q_1,3q_1,4q_1,5q_1$ & $1,1,\frac{1}{2},\frac{1}{3},\frac{1}{4},-\frac{1}{10}$ & $3q_1,5q_1$ & $\geq 7$\\
    \hline
    2 & $q_1,-,q_1,2q_1,3q_1,3q_1$ & $1,0,1,1,1,-2$ & $2q_1,3q_1$  & $\geq 5$ \\
    \hline
%    3 & $q_1,-,q_1,2q_1,3q_1,4?q_1$ & $1,0,1,1,1,d_6$ & $2q_1,3q_1$  & $=2$ \\
%    \hline
    3 & $q_1,-,q_1,2q_1,3q_1,-$ & $1,0,1,1,1,0$ & $2q_1,3q_1$  & $=2$ \\
    \hline
    4 & $q_1,-,-,q_1,2q_1,q_1$ & $1,0,0,1,\frac{3}{2},d_6$ & $q_1,q_1$  & $\geq 5$ \\
    \hline
    5 & $q_1,-,-,q_1,2q_1,-$ & $1,0,0,1,\frac{3}{2},0$ & $q_1,q_1$  & $\geq 5$ \\
    \hline
    6 & $q_1,-,-,-,2q_1,q_1$ & $1,0,0,0,1,d_6$ & $q_1,q_1$ & $=2$ \\
    %\hline
    %7' & $q_1,-,-,-,3q_1,3q_1$ & $1,0,0,0,1,d_6$ & $2q_1,2q_1$ & $=3$ \\
    \hline
    7 & $q_1,-,-,-,q_5,-$ & $1,0,0,0,1,0$ & $q_5-q_1,2q_5-3q_1$ & any \\
    \hline
    8 & $q_1,-,-,-,3q_1,3q_1$ & $1,0,0,0,1,d_6$ & $2q_1,3q_1$ & $=3$ \\
    \hline
    9 & $q_1,-,-,-,-,q_6$ & $1,0,0,0,0,1$ & $\frac{1}{2}(q_1+q_6),q_6$  & any \\
    \hline
    10 & $-,q_2,q_2,q_2,q_2,2q_2$ & $0,1,1,d_4,d_5,\frac{1}{2}d_5$ & $q_2,2q_2$ & $=3$ \\
    \hline
    11 & $-,q_2,q_2,q_2,q_2,2q_2$ & $0,1,1,d_4,d_4,d_6$ & $q_2,2q_2$ & $=2$ \\
    \hline
    12 & $-,q_2,q_2,q_2,q_2,-$ & $0,1,1,d_4,3d_4,0$ & $q_2,2q_2$ & $\neq 3$ \\
    \hline
    13 & $-,q_2,q_2,q_2,q_2,2q_2$ & $0,1,1,d_4,d_5,\frac{1}{2}(d_5-3d_4)$ & $q_2,2q_2$ & $\geq 5$ \\
    \hline
    % & $-,q_2,q_2,q_2,-,-$ & $0,1,1,d_4,0,0$ & $q_2,2q_2$ & $=3$ \\
    %\hline
    14 & $-,q_2,q_2,q_2,-,2q_2$ & $0,1,1,d_4,0,-\frac{3}{2}d_4$ & $q_2,2q_2$ & $\geq 5$ \\
    \hline
    15 & $-,q_2,q_2,-,-,2q_2$ & $0,1,1,0,0,d_6$  & $q_2,2q_2$ & $=2$ \\
    \hline
    16 & $-,2q_3,q_3,-,-,q_3$ & $0,1,1,0,0,d_6$  & $0,q_3$ & $=2$ \\
    \hline
    17 & $-,q_2,-,q_2,q_2,2q_2$ & $0,1,0,1,d_5,\frac{1}{2}d_5$ & $q_2,2q_2$ & $\geq 3$  \\
    \hline
    18 & $-,q_2,-,q_2,-,2q_2$ & $0,1,0,1,0,d_6$ & $q_2,2q_2$ &$=2$   \\
    \hline
    19 & $-,3q_4,-,q_4,-,3q_4$ & $0,1,0,1,0,d_6$ & $q_4,3q_4$ &$=3$   \\
    \hline
    20 & $-,q_2,-,-,q_2,2q_2$ & $0,1,0,0,1,\frac{1}{2}$ & $q_2,2q_2$ & $\geq 3$  \\
    \hline
    21 & $-,q_2,-,-,-,q_6$ & $0,1,0,0,0,1$ & $\frac{1}{3}(2q_6-q_2),q_6$ & any  \\
 \hline
  \end{tabular}
  }
\caption{$H$ in $G_2(k)$}
\label{tab:G2_cases}
\end{table}
%\FloatBarrier

\medbreak
\noindent{\textbf{Cases for $G_2(k)$:}} 
Imposing the condition that $u:k\rightarrow U_H$ is a homomorphism of groups, $u(a+b)=u(a)u(b)$, we obtain the following equations:
  \begin{equation}
    \label{eq:add_sys_g2}
    \left\{
    \begin{aligned}
      &c_1(a+b)^{q_1}
      =
      c_1a^{q_1}+c_1b^{q_1}
      \\
      &c_2(a+b)^{q_2}
      =
      c_2a^{q_2}+c_2b^{q_2}
      \\
      &c_3(a+b)^{q_3}
      =
      c_3a^{q_3}+c_3b^{q_3}+c_1c_2a^{q_2}b^{q_1}
      \\
      &c_4(a+b)^{q_4}
      =
      c_4a^{q_4}+c_4b^{q_4}+c_1^{2}c_2a^{q_2}b^{2q_1}+2c_1c_3a^{q_3}b^{q_1}
      \\
      &c_5(a+b)^{q_5}
      =
      c_5a^{q_5}+c_5b^{q_5}+c_1^{3}c_2a^{q_2}b^{3q_1}+3c_1^{2}c_3a^{q_3}b^{2q_1}
      %\\
      %&\qquad
      +3c_4c_1a^{q_4}b^{q_1}
      \\
      &c_6(a+b)^{q_6}
      =
      c_6a^{q_6}+c_6b^{q_6}
      -c_1^3c_2^2a^{2q_2}b^{3q_1}
      +c_2^2c_1^{3}a^{q_2}b^{3q_1+q_2}
      \\
      &\qquad
      -3c_1^2c_2c_3a^{q_2+q_3}b^{2q_1}
      -3c_1^2c_2c_3a^{q_2}b^{2q_1+q_3}
      +3c_1^2c_2c_3a^{q_3}b^{2q_1+q_2}
      \\
      &\qquad
      -3c_1c_3^2a^{2q_3}b^{q_1}
      -6c_1c_3^2a^{q_3}b^{q_1+q_3}
      +3c_1c_2c_4a^{q_4}b^{q_1+q_2}
      \\
      &\qquad
      -3c_4c_3a^{q_4}b^{q_3}
      +c_2c_5a^{q_5}b^{q_2}.
    \end{aligned}
    \right.
  \end{equation}
  \medbreak
  \noindent{\textbf{Case 1:}}
  Consider System \eqref{eq:add_sys_g2} with $c_1,c_2\neq 0$. By Lemma~\ref{cr_is_one}, we may assume that $c_1=c_2=1$, and $q_1, q_2$ are $p$-powers. The third equation in \eqref{eq:add_sys_g2} and Lemma~\ref{lemma0}.(1) show that $p\neq 2$, $c_3=1/2$ and $q_1=q_2=q_3/2$. Then, by the fourth equation and Lemma~\ref{lemma0}.(2), we have $p\neq 3$, $q_4=3q_1$ and $c_4=1/3$. Hence, by Lemma~\ref{lemma0}.(3), we must have $c_5=1/4$ and $q_5=4q_1$. Thus, the last two equations are
$$
\left\{
\begin{aligned}
  &\frac{1}{4}(a+b)^{q_5}
  =
  \frac{1}{4}a^{q_5}+\frac{1}{4}b^{q_5}+a^{q_1}b^{3q_1}+\frac{3}{2}a^{2q_1}b^{2q_1}+a^{3q_1}b^{q_1}
  \\
  &
  c_6(a+b)^{q_6}
  =
  c_6a^{q_6}+c_6b^{q_6}
  -\frac{1}{2}a^{q_1}b^{4q_1}
  -a^{2q_1}b^{3q_1}
  -a^{3q_1}b^{2q_1}
  -\frac{1}{2}a^{4q_1}b^{q_1}.
\end{aligned}
\right.
$$
By Lemma~\ref{lemma0}.(4), $p\neq 5$, $q_6=5q_1$ and $c_6=-1/10$. Then, imposing the compatibility condition with the map $t:k^{\times}\rightarrow T_H$, we get $m_1=3mq_1$, $m_2=5mq_1$.
\medbreak
\noindent{\textbf{Case 2:}} Let $c_1,c_3\neq 0$, $c_2=0$ and $p\neq 2$. We may assume that $c_1=c_3=1$ by Lemma~\ref{cr_is_one}, and
$q_1,q_3$ are $p$-powers. By Lemma~\ref{lemma0}.(1), $q_1=q_3$, $q_4=2q_1$ and $c_4=1$ and System \eqref{eq:add_sys_g2} is
$$
\left\{
\begin{aligned}
  &(a+b)^{q_1}
  =
  a^{q_1}+b^{q_1}
  \\
  &(a+b)^{2q_1}
  =
  a^{2q_1}+b^{2q_1}+2a^{q_1}b^{q_1}
  \\
  &c_5(a+b)^{q_5}
  =
  c_5a^{q_5}+c_5b^{q_5}+3a^{q_1}b^{2q_1}+3a^{2q_1}b^{q_1}
  \\
  &c_6(a+b)^{q_6}
  =
  c_6a^{q_6}+c_6b^{q_6}
  -3a^{2q_1}b^{q_1}
  -6a^{q_1}b^{2q_1}
  -3a^{2q_1}b^{q_1}.
\end{aligned}
\right.
$$
If $p> 3$, by Lemma~\ref{lemma0}.(2) we have $c_5=1$, $q_5=3q_1$. Thus, by Lemma~\ref{lemma0}.(2), we have $c_6=-2$ and $q_6=3q_1$. The compatibility with $T_H$ gives $m_1=2mq_1$, $m_2=3mq_1$. If $p=3$, we have four cases according to whether or not $c_5$ and $c_6$ are zero. Moreover, there is an exceptional isogeny $\phi$ interchanging the long and the short root groups. If $c_5,c_6\neq 0$ then $U_H$ maps under $\phi$ to a corresponding group described in Case 10. The two cases $(c_5=0,c_6\neq 0)$ and $(c_5\neq 0,c_6= 0)$ are conjugate by an element of $N_G(T)$ corresponding to a reflection $s_{\alpha_2}$. Consider the case $(c_5=0,c_6\neq 0)$ and notice that by System \eqref{eq:add_sys_g2}, $q_6$ is also a power of $p=3$. Notice also, that the compatibility with $T_H$ gives $q_6=3q_1$. Thus, in this case $U_H$ is mapped under $\phi$ to a corresponding group described in Case 17. The case $(c_5=0,c_6=0)$ is covered by Case 20 if we apply $\phi$.
\medbreak
\noindent{\textbf{Case 3:}} Let $c_1,c_3\neq 0$, $c_2=0$ and $p=2$. As before, we may assume that $c_1=c_3=1$ and $q_1,q_3$ are $p$-powers. System \eqref{eq:add_sys_g2} is
$$
\left\{
\begin{aligned}
  &(a+b)^{q_1}
  =
  a^{q_1}+b^{q_1}
  \\
  &(a+b)^{q_3}
  =
  a^{q_3}+b^{q_3}
  \\
  &c_4(a+b)^{q_4}
  =
  c_4a^{q_4}+c_4b^{q_4}
  \\
  &c_5(a+b)^{q_5}
  =
  c_5a^{q_5}+c_5b^{q_5}+3a^{q_3}b^{2q_1}+3c_4a^{q_4}b^{q_1}
  \\
  &c_6(a+b)^{q_6}
  =
  c_6a^{q_6}+c_6b^{q_6}
  -3a^{2q_3}b^{q_1}
  -3c_4a^{q_4}b^{q_3}.
\end{aligned}
\right.
$$
By Lemma~\ref{lemma0}.(5), we get $q_4=2q_1$, $q_3=q_1$, $c_4=1$, $c_5=1$ and $q_5=3q_1$. Hence, if $c_6=0$, the compatibility with $T_H$ yields $m_1=2mq_1$, $m_2=3mq_1$. If $c_6\neq 0$, by the last equation of the above system, $q_6$ has to be a power of $p=2$, however, the compatibility with $T_H$ implies $q_6=3q_1$, a contradiction.
\medbreak
\noindent{\textbf{Cases 4,5:}} Assume now that $c_1,c_4\neq 0$ and $c_2=c_3=0$. Then $q_1$ and $q_4$ are $p$-powers. With Lemma~\ref{cr_is_one} we may assume that $c_1=c_4=1$ and System \eqref{eq:add_sys_g2} is
$$
\left\{
\begin{aligned}
  &(a+b)^{q_1}
  =
  a^{q_1}+b^{q_1}
  \\
  &(a+b)^{q_4}
  =
  a^{q_4}+b^{q_4}
  \\
  &c_5(a+b)^{q_5}
  =
  c_5a^{q_5}+c_5b^{q_5}+3a^{q_4}b^{q_1}
  \\
  &c_6(a+b)^{q_6}
  =
  c_6a^{q_6}+c_6b^{q_6}.
\end{aligned}
\right.
$$
First assume that $p\neq 3$. Then by Lemma~\ref{lemma0}.(1), $p\neq 2$, $c_5=3/2$, $q_4=q_1$, $q_5=2q_1$ and $c_6$ is arbitrary. If $c_6\neq 0$, we get $m_1=m_2=mq_1$ and $q_6=q_1$ from the compatibility with $T_H$, while for $c_6=0$ we have $m_1=m_2=mq_1$.

Next consider the case where $p=3$. Then $q_5$ and $q_6$ are powers of $3$ if $c_5\neq 0$ and $c_6\neq 0$ respectively. Using an exceptional isogeny $\phi$, we see that the cases $(c_5=0,c_6\neq 0)$ and $(c_5=0,c_6=0)$ follow from Case 19 and Case 21 respectively. Indeed, case $(c_5=0,c_6=0)$ is easily checked while for $(c_5=0,c_6\neq 0)$, using the compatibility with $T_H$ and Lemma~\ref{lemmap}, it follows that $q_1=q_4=q_6$.
For $(c_5\neq 0,c_6=0)$ and $(c_5\ne 0,c_6\neq 0)$, the compatibility with the torus gives $q_5=q_1+q_4$. Hence $q_5=2q_1=2q_4$, a contradiction with $p=3$.
\medbreak
\noindent{\textbf{Cases 6--9:}} Let $c_1\neq 0$ and $c_2=c_3=c_4=0$. Assume first that $c_5\neq 0$. By Lemma~\ref{cr_is_one}, we may assume that $c_1=c_5=1$. Considering System \eqref{eq:add_sys_g2}, we see that $q_1$ and $q_5$ are powers of $p$. If $c_6=0$, the compatibility with $T_H$ gives $m_1=m(q_5-q_1)$, $m_2=m(2q_5-3q_1)$.
 Assume now that $c_6\neq 0$. Then $q_6$ is also a power of $p$ and from the compatibility with $T_H$ we obtain
$$
m_1=\frac{m}{2}(q_1+q_6)=\frac{m}{3}(q_5+q_6)
,\quad
m_2=mq_6
\quad\text{and}\quad
2q_5=3q_1+q_6.
$$
If $q_1<q_6$ then $q_5=q_1\frac{3+p^{f_6-f_1}}{2}$. By Lemma~\ref{lemmap} we have $q_6=pq_1$ and hence $2p^{f_5-f_1}=3+p$. Thus $p\neq 2$ and $p^{f_5-f_1}=\frac{3+p}{2}\leq p$. Since $q_5>q_1$, i.e. $f_5-f_1\geq 1$, we must have $p=3$ and $f_5=f_1+1$. We obtain the case $q_5=q_6=3q_1$, $p=3$ and $m_1=2mq_1$, $m_2=3mq_1$. %In this case $U_H$ is conjugate by an exceptional isogeny to $U_H$ from Case 12.
If $q_1>q_6$ it follows from Lemma~\ref{lemmap} that there is no solution. If $q_1=q_6$ then $q_5=2q_1=2q_6$. Thus $p=2$ and $m_1=m_2=mq_1$.

Now let $c_5=0$. We may assume that $c_6=1$ since we exclude root groups. In this case $m_1=m(q_1+q_6)/2$, $m_2=mq_6$.
\medbreak
\noindent{\textbf{Cases 10--13:}} Consider System \eqref{eq:add_sys_g2} with $c_1=0$ and $c_2,c_3\neq 0$. Then $q_2,q_3$ are $p$-powers and by Lemma~\ref{cr_is_one} we may assume that $c_2=c_3=1$. Suppose first that $c_4\neq 0$ and $c_5\neq 0$. If $p=3$, by Lemma~\ref{lemma0}.(1), $q_5=q_2$, $q_6=2q_2$ and $c_6=c_5/2$. In particular $c_6\neq 0$. From the compatibility with $T_H$ we get $m_1=mq_2$, $m_2=2mq_2$ and $q_4=q_3=q_2$. If $p\neq 3$, by Lemma~\ref{lemma0}.(6), we have three cases:

(I) $q_3=q_2$ and $3c_4=c_5$. Then $q_2=q_3$, $q_4=q_5$, $q_6$ is a power of $p$ and $c_6$ is arbitrary. If $c_6\neq 0$, from the compatibility with $T_H$, we get $m_1=mq_2$, $m_2=2mq_2$, $q_2=q_4$
and $q_6=2q_2$ and hence $p=2$ since $q_6$ is also a power of $p$. 

(II) $q_3=q_2$ and $3c_4\neq c_5$. Then $p\geq 5$, $q_2=q_3=q_4=q_5$ and $q_6=2q_2$ so $c_6=(c_5-3c_4)/2\neq 0$. From the compatibility with $T_H$ we get $q_6=2q_2$, $m_1=mq_2$, $m_2=2mq_2$. 

(III) $q_3\neq q_2$. With Lemma~\ref{lemma0} we have $q_5=q_3\neq q_2=q_4$ and the compatibility with $T_H$ gives $q_3=q_2$, a contradiction.
%(III) $q_3\neq q_2$. Then $q_5=q_3\neq q_2=q_4$ and $q_6=q_2+q_3$ with the restrictions coming from Lemma~\ref{lemma0}.(6). In particular $c_6\neq 0$. By \eqref{Pi_equations}, we get $q_3=q_2$, a contradiction.
\medbreak
\noindent{\textbf{Case 14:}} Continuing the discussion with $c_1=0$ and $c_2,c_3\neq 0$, let $c_4\neq 0$ and $c_5=0$. First assume that $p=3$. If $c_6\neq0$, by System \eqref{eq:add_sys_g2}, $q_6$ is a power of $p$. Then, from the compatibility with $T_H$ we get $m_1=mq_4$, $m_2=m(q_3+q_4)$ and $q_6=q_3+q_4=(q_2+3q_4)/2$.
Thus $2\mid q_6=q_3+q_4$, but $p=3$, a contradiction. Therefore $c_6=0$ and from the compatibility with $T_H$ we get $m_1=mq_4$, $m_2=m(q_3+q_4)$ and $2q_3=q_2+q_4$. By Lemma~\ref{lemmap}, we cannot have $q_2$ distinct from $q_4$ since $(3^f+1)/2$ is not a power of $3$ for $f\geq1$. Hence $q_2=q_4=q_3$. %However, with Lemma~\ref{lemmap} we get a contradiction since $p=3$ and $(3^f+1)/2$ is not a power of $3$ for $f\geq1$.
One checks that $U_H$ is mapped by an exceptional isogeny to a corresponding group described in Case 8.

Now let $p\neq 3$. Applying Lemma~\ref{lemma0}.(1) we have $q_4=q_3$, $q_6=2q_3$, $p\neq 2$ and $c_6=-\frac{3}{2}c_4$. From the compatibility with $T_H$ we obtain $m_1=mq_3$, $m_2=2mq_3$, $q_2=q_3$.
\medbreak
\noindent{\textbf{Cases 15,16:}}
Let $c_1=c_4=0$ and $c_2,c_3\neq 0$, so that as usual $q_2,q_3$ are $p$-powers. By Lemma~\ref{cr_is_one}, we may assume that $c_2=c_3=1$. If $c_5\neq 0$, by Lemma~\ref{lemma0}.(1), $p\neq 2$, $q_5=q_2$, $q_6=2q_2$ and $c_6=c_5/2$. In particular $c_6\neq 0$. The compatibility with $T_H$ shows that also $q_2=q_3$. In this case $U_H$ is conjugate by a representative in $N_G(T)$ of $s_{\alpha_1}$ to $U_H$ from Case 17.

If $c_5=c_6=0$ then the corresponding group $U_H$ is conjugate by an element of $N_G(T)$ to the group $U_H$ in Case 7.

Finally, suppose  $c_5=0$ and $c_6\neq 0$, so that $q_6$ is also a power of $p$. From the compatibility with $T_H$ we get $m_1=m(q_6-q_3)$, $m_2=mq_6$, $3q_3=q_2+q_6$. Thus $p=2$, $q_2\ne q_6$,  and either $q_6=2q_2=2q_3$ or $q_2=2q_6=2q_3$. To see this, let $q_i=2^{f_i}$. If $q_2>q_6$ then $3q_3=q_6(1+2^{f_2-f_6})$. 
By Lemma~\ref{lemmap}, the only solution is $f_2-f_6=1$ so $q_2=2q_6$, hence $q_3=q_6$. We get $m_1=0$ and $m_2=mq_3$. If $q_2<q_6$ then $3q_3=q_2(2^{f_6-f_2}+1)$ and the argument is similar: $q_6=2q_2$, hence $q_3=q_2$. We get $m_1=mq_2$ and $m_2=2mq_2$.

\medbreak
\noindent{\textbf{Cases 17,18,19:}}
Consider System~\eqref{eq:add_sys_g2} with $c_1,c_3=0$ and $c_2,c_4\neq 0$, so that $q_2$ and $q_4$ are $p$-powers. By Lemma~\ref{cr_is_one} we may assume that $c_2=c_4=1$. If $c_5\neq 0$, by Lemma~\ref{lemma0}.(1), $q_5=q_2$, $q_6=2q_2$ (so $p\neq 2$) and $c_6=c_5/2$. From the compatibility with $T_H$ we get $m_1=mq_2$, $m_2=2mq_2$ and $q_4=q_2$.

If $c_5, c_6=0$ then the corresponding group $U_H$ is conjugate by an element of $N_G(T)$ to the group $U_H$ in Case 9. If $c_5=0$ and $c_6\neq 0$,
then $q_6$ is a $p$-power, and from the compatibility with $T_H$ we get $m_1=mq_4$, $m_2=mq_6$ and $q_2=2q_6-3q_4$.
If $q_2=q_4$ then $2q_6=4q_2=4q_4$ hence $p=2$, $q_6=2q_2=2q_4$, $m_1=mq_4$, $m_2=2mq_4$. If $q_4>q_2$ then $2q_6=q_2(1+3p^{f_4-f_2})$ and by Lemma~\ref{lemmap} this equation has no solution. If $q_2>q_4$ then $2q_6=q_4(p^{f_2-f_4}+3)$. By Lemma~\ref{lemmap} the only solution is $f_2-f_4=1$. Then $p=3$ and $q_2=q_6=3q_4$.
\medbreak
\noindent{\textbf{Cases 20,21:}}
If $c_1,c_3,c_4=0$ and $c_2,c_5\neq 0$, then $q_2$ and $q_5$ are $p$-powers, and by Lemma~\ref{cr_is_one} we may assume that $c_2,c_5=1$. Again, by Lemma~\ref{lemma0}.(1), $q_5=q_2$, $q_6=2q_2$ (so $p\neq 2$) and $c_6=1/2$. From the compatibility with $T_H$ we get $m_1=mq_2$, $m_2=2mq_2$.

If $c_1,c_3,c_4,c_5=0$ and $c_2,c_6\neq 0$, by Lemma~\ref{cr_is_one}, we may assume that $c_2,c_6=1$. From the compatibility with $T_H$ we get $m_1=m(2q_6-q_2)/3$, $m_2=mq_6$.
\end{proof}

\begin{prop}
  \label{prop:non-existence}
  Let $G$ be a simple algebraic group of rank $2$ defined over an algebraically closed field of characteristic $p>0$. If $H$ is a  closed subgroup of $G$, with $\dim H\leq 2$, then $H$ is not epimorphic in $G$. 
%  Let $H$ be a $2$-dimensional closed subgroup of a rank $2$ simple algebraic group $G$. Then $H$ is not epimorphic in $G$. 
\end{prop}
\begin{proof}
  By \S\ref{subsec:criteria}.(7), Corollary~\ref{cor:up_to_isogeny}.(1) and Lemma~\ref{lem:mixed_2_dim}, we may assume that $G$ is simply connected and $H=H^\circ$, with $H=U_H\rtimes T_H$, non-abelian, with $U_H$ $1$-dimensional unipotent and $T_H$ a torus. By Lemma~\ref{lem:root_group_case} and Lemma~\ref{lem:structure} we may assume that $H$ is one of the groups described in Tables~\ref{tab:A2_cases},~\ref{tab:B2_cases} and~\ref{tab:G2_cases}. The proof is by considering in turn each of the cases in these tables.
  
  For a $G$-module $V$, an element $v\in V$ and a nonzero integer $a$, we denote by $\overline{\otimes^a v}$ the image of the $a$-fold tensor product $\otimes^a v$ in the symmetric power $S^a(V)$. Moreover, we write $f_{\alpha}$ for the Chevalley basis vector of the Lie algebra of $G$ corresponding to the negative root $\alpha\in\Phi$.
  \smallbreak
  \noindent{\textbf{Cases for $\SL_3(k)$:}}
  We fix  an ordered basis $(e_1,e_2,e_3)$ of the natural $G$-module $V$, so that the matrices of $u_{\alpha_1}(x)$, $u_{\alpha_2}(x)$ and $u_{\alpha_1+\alpha_2}(x)$, are $I+xE_{12}$, $I+xE_{23}$ and $I+xE_{13}$ respectively.
  \smallbreak
  \noindent{\bf{Case 1:}} Here it is straightforward to check that $H$ is conjugate to the image of a Borel subgroup of ${\rm SL}_2(k)$ under the two-dimensional
  irreducible representation of highest weight $2q_1$. Then by Corollary~\ref{cor:reductive_proper},
  $H$ is not epimorphic.
  \smallbreak
  \noindent{\bf{Case 2:}} Here the torus $T_H$ acts with weights $\frac{m}{3}(q_1+q_3)$, respectively $\frac{m}{3}(q_3-2q_1)$, $\frac{m}{3}(q_1-2q_3)$ on the given basis vectors. If $q_1=q_3$, then $e_1\otimes (e_2-e_3)\otimes (e_2-e_3)\in \otimes^3(V)$ is fixed by $U_H$ and of zero weight for $T_H$ but is not a weight vector for $T$.
  If $q_1>q_3$, so that $q_1\geq 2q_3$, we first suppose $q_1=2q_3$, in which case, $e_1\wedge e_2\in\wedge^2(V)$ is fixed  by $H$, but not by $T$. Next suppose $q_1>2q_3$. Set $a = q_1-2q_3$ and $b = 2q_1-q_3$, both strictly positive integers. Then set $Y = \wedge^2(V)$, and note that $e_1\wedge e_2\in Y$ is fixed  by $U_H$ and has $T_H$-weight $-\frac{m}{3}a$, and $(e_1\wedge e_3)\in Y$ is fixed by $U_H$ and has $T_H$-weight $\frac{m}{3}b$. Thus the vector $w=\overline{\otimes^a(e_1\wedge e_3)}\otimes \overline{\otimes ^b(e_1\wedge e_2)}$ in $S^a(Y)\otimes S^b(Y)$ is fixed by $H$, but is not fixed by $T$. Indeed, the cocharacter $\alpha_1^{\vee}$ has strictly positive weight on $w$. The case $q_3>q_1$ is entirely similar.

  So in all cases, by Theorem~\ref{thm:epi_characterization}.(4), $H$ is not epimorphic in $G$.
  \medbreak
  \noindent{\textbf{Cases for $\Sp_4(k)$:}}
  We fix as follows ordered bases of $V_1$ and $V_2$, the Weyl modules for $G$ of highest weight $\omega_i$, for $i = 1, 2$:
  %highest weight $\lambda_i$, for $i=1,2$:
\begin{flalign*}
  V_2:\, & v_{21},\, v_{22}=f_{\alpha_2}v_{21},\, v_{23}= f_{\alpha_1+\alpha_2}v_{21},\, v_{24}=f_{\alpha_1+2\alpha_2}v_{21};
  &&\\
  V_1:\, & v_{11},\, v_{12}=f_{\alpha_1}v_{11},\, v_{13}= f_{\alpha_1+\alpha_2}v_{11},\, v_{14}=f_{\alpha_1+2\alpha_2}v_{11},\,v_{15}=f_{\alpha_1}f_{\alpha_1+2\alpha_2}v_{11}.
  %&&\\
\end{flalign*}
(Here $v_{11}$ and $v_{21}$ are highest weight vectors in the respective modules.) We use the matrix expressions of the root groups with respect to these bases.
%which are not difficult to deduce.
%It is not difficult to find the matrix expression of the root groups with respect to these bases.

%Then, one checks that the root groups for $G$ act as follows:
%\noindent On $L(\lambda_2)$, $x_{\alpha_1}(t) = I-tE_{23}$,  $x_{\alpha_2}(t) = I+t(E_{12}+E_{34})$,
%  $x_{\alpha_1+\alpha_2}(t) = I+t(E_{13}-E_{24})$, 
%  $x_{\alpha_1+2\alpha_2}(t) = I+tE_{14}$;

%  \noindent on $L(\lambda_1)$, $x_{\alpha_1}(t) = I+t(E_{12}+E_{45})$,  $x_{\alpha_2}(t) = I+t(2E_{23}+E_{34})+t^2E_{24}$,
%  $x_{\alpha_1+\alpha_2}(t) = I+t(2E_{13}-E_{35})-t^2E_{15}$, 
% $x_{\alpha_1+2\alpha_2}(t) = I+t(E_{14}+E_{25})$.
\FloatBarrier
\renewcommand{\arraystretch}{1.5}
\begin{table}[ht]
  \begin{tabular}{| c | l | c |}
    \hline
    Cases& $w$ & $W$ \\
    \hline
    3 & $c_4(v_{22}\wedge v_{23}) - v_{22}\wedge v_{24}$ & $\wedge^2(V_2)$ \\
    \hline
    4 & $(q_1=2q_3)$ $v_{21}$& $V_2$\\
    & ($q_1=q_3$) $v_{14}$& $V_1$\\
    & $(q_1>2q_3)$ $\overline{\otimes^{2q_1-2q_3}v_{21}}\otimes \overline{\otimes^{q_1-2q_3}v_{14}}$& $S^{2q_1-2q_3}(V_2)\otimes S^{q_1-2q_3}(V_1)$\\
    & $(q_1<q_3$) $\overline{\otimes^{2q_3-2q_1}v_{21}}\otimes \overline{\otimes^{2q_3-q_1}v_{14}}$&    $S^{2q_3-2q_1}(V_2)\otimes S^{2q_3-q_1}(V_1)$           \\
    \hline
    5,6 & $v_{22}-v_{23}$ & $V_2$ \\
    \hline
  \end{tabular}
  \caption{}
  \label{tab:B2_fixed_point}
\end{table}
\FloatBarrier
%Note that in Case 4, the condition $q_1<2q_3$ implies that $q_1\leq q_3$. 
\medbreak
\noindent{\bf{Case 1:}} As in Case 1 for ${\rm SL}_3(k)$, one can check that the group $H$ is conjugate to the image of a Borel subgroup of ${\rm SL}_2(k)$ acting on the irreducible module of highest weight $3q_1$, and conclude by applying Corollary~\ref{cor:reductive_proper}.
\medbreak
\noindent{\bf{Case 2:}} Here $H$ lies in the maximal rank reductive subgroup of type $A_1\tilde{A_1}$, and we apply again Corollary~\ref{cor:reductive_proper}.
\medbreak
\noindent{\bf{Cases 3,4,5,6:}} In each case we identify a module $W$, a fixed point $w\in W$ for $H$ which is not fixed by $T$ and then apply Theorem~\ref{thm:epi_characterization}.(4) to  see that $H$ is not epimorphic in $G$. The information is given in Table~\ref{tab:B2_fixed_point}. For Case 4, notice that $\alpha_1^{\vee}$ or $\alpha_2^{\vee}$ has strictly  positive weight on $w$.
\medbreak
\noindent{\textbf{Cases for $G_2(k)$:}}
Here we use the explicit matrix representations of root elements
acting on $V$, the $7$-dimensional Weyl module of highest weight $\omega_1$.
We use the structure constants from \textsc{GAP} \cite{GAP},
the same as those used to produce the expressions in Table~\ref{tab:G2_cases},
and fix the following ordered basis $\mathcal B$ of $V$:
$$
\begin{aligned}
  &
  v_1,\quad
  v_2 = f_{\alpha_1}v_1,\quad
  v_3 = f_{\alpha_1+\alpha_2}v_1,\quad
  v_4=f_{2\alpha_1+\alpha_2}v_1,\quad
  \\
  &
  v_5=f_{3\alpha_1+\alpha_2}v_1
  ,\quad
  v_6=f_{3\alpha_1+2\alpha_2}v_1
  ,\quad
  v_7=f_{\alpha_1}f_{3\alpha_1+2\alpha_2}v_1.
\end{aligned}
$$
where $v_1$ is a maximal vector for the fixed Borel subgroup of $G$. Then the root groups are as follows:
$$
\begin{aligned}
  &u_{\alpha_1}(x) = I+ x(E_{12}+2E_{34}+E_{45}+E_{67})+x^2E_{35},\\
  &u_{\alpha_2}(x) = I+x(-E_{23}+E_{56}),\\
  &u_{\alpha_1+\alpha_2}(x) = I+ x(E_{13}-2E_{24}-E_{46}+E_{57})+x^2E_{26},\\
  &u_{2\alpha_1+\alpha_2}(x) = I+ x(2E_{14}-E_{25}+E_{36}-E_{47})-x^2E_{17},\\
  &u_{3\alpha_1+\alpha_2}(x) = I+x(E_{15}+E_{37}),\\
  &u_{3\alpha_1+2\alpha_2}(x) = I+x(E_{16}+E_{27}).
\end{aligned}
$$
%As for $\Sp_2(k)$, we indicate a module $W$, a fixed point $w\in W$ for $H$ which is not fixed by $T$ and then apply Theorem~\ref{thm:epi_characterization}.(4) to conclude that $H$ is not epimorphic in $G$. The information is given in Table~\ref{tab:B2_fixed_point}.
%We exemplify the method on the most involved cases.
\medbreak
\noindent{\bf{Case 1:}}
We argue that $T_HU_H$ lies in a principal $A_1$-subgroup of $G$, and then apply Corollary~\ref{cor:reductive_proper}.
We first show that $U_H$ lies in such a subgroup. The argument is similar to the proof of \cite[Theorem 3]{Testerman1}.
It is known that the principal $A_1$-subgroup in $G_2$ acts irreducibly on the module $V$, when $p\geq 7$, and $V$ is isomorphic to the irreducible
${\rm SL}_2(k)$-module of highest weight $6$. Using the construction of the irreducible representation coming from the action of ${\rm SL}_2(k)$ on the
homogeneous polynomials of degree $6$ in $k[X,Y]$, with the basis $w_i = X^{6-i}Y^i$, $0\leq i\leq 6$, one can check that the action of $u(t)$
on the basis $\{\gamma_iw_i, 1\leq i\leq 7\}$, where we take $(\gamma_1,\gamma_2,\gamma_3,\gamma_4,\gamma_5,\gamma_6,\gamma_7) =
(1,1,-2,-3,-12,-60,-360)$, is precisely the action of $\begin{pmatrix}1&t^{q_1}\cr 0&1\end{pmatrix}$ on the basis $\{v_i,1\leq i\leq 7\}$.
  Moreover, since the action of the toral element
  $t(\lambda) = \alpha_1^\vee(\lambda^{3q_1m})\alpha_2^\vee(\lambda^{5q_1m})$
  on the basis $\gamma_iw_i$ is $\lambda^{q_1m(4-i)}$, for $1\leq i\leq 7$,
  and setting $\mu^2 = \lambda^m$, we have that
  the action of the semisimple element $t(\mu)$
  is precisely that of
  ${\rm diag}(\mu^{q_1},\mu^{q_1})$
  %$\begin{pmatrix} \mu^{q_1} &0\cr 0&\mu^{q_1}\end{pmatrix}$
  on the basis $\mathcal B$.

  \medbreak
  \noindent{\bf{Cases 9,20,21:}}
  Here $H$ lies in a maximal rank subsystem subgroup
  of type $A_1A_1$, $A_2$ and $A_2$ respectively.
  By Corollary~\ref{cor:reductive_proper},
  $H$ is not epimorphic.
  \smallbreak
  
  \FloatBarrier
\renewcommand{\arraystretch}{1.3}
\begin{table}[ht]
  \begin{tabular}{| c | l |}
    \hline
    Cases & weights for $m=1$ (here $q=q_1=q_2=q_3=q_4$)\\
    \hline
    2,3,8 & $2q,q,q,0,-q,-q,-2q$  \\
    \hline
    4,5,6 & $q,0,q,0,-q,0,-q$  \\
    \hline
    7 & $q_5-q_1,q_5-2q_1,q_1,0,-q_1,2q_1-q_5,q_1-q_5$  \\
    %\hline
    %9&$2q,q,q,0,-q,-q,-2q$\\
    \hline
    10--15,17,18 & $q,q,0,0,0,-q,-q$  \\
    \hline
    16 & $0,q,-q,0,q,-q,0$\\
    \hline
    19 & $q,2q,-q,0,q,-2q,-q$\\
    \hline
  \end{tabular}
  \caption{}
  \label{tab:G2_weights}
\end{table}
\FloatBarrier

\FloatBarrier
\renewcommand{\arraystretch}{1.3}
\begin{table}[ht]
  \begin{tabular}{| c | l | c |}
    \hline
   Cases& $w$ & $W$ \\
   \hline
  2,3 & $(v_2-v_3)\wedge v_4\wedge v_6$ & $\wedge^3(V)$ \\
%  \hline
%  3 & $v_1\wedge v_2\wedge (v_5+v_6)\wedge v_7$ & $\wedge^4(V)$\\
%    \hline
%  4 & $(v_2+v_3)\wedge v_4\wedge v_6$ & $\wedge^3(V)$ \\
  \hline
  4,5 & $2(c_6-1)v_2+v_4-2v_6$ & $V$ \\
  \hline
  6 & $c_6v_2-v_6$ & $V$\\
  \hline
  7 & $(q_1>q_5)$ $\overline{\otimes^{2q_1-q_5}v_1}\otimes \overline{\otimes^{q_1-q_5} v_6}$ & $S^{2q_1-q_5}(V)\otimes S^{q_1-q_5}(V)$\\
  & $(q_1=q_5)$ $v_1$ & $V$\\
  & $(q_5=2q_1)$ $v_6$&$V$\\
  & $(q_1<q_5)$ $\overline{\otimes^{q_5-2q_1}v_1}\otimes \overline{\otimes^{q_5-q_1} v_6}$ & $S^{q_5-2q_1}(V)\otimes S^{q_5-q_1}(V)$\\
  \hline
  8&$(c_6v_5-v_6)\wedge v_3\wedge v_4$&$\wedge^3(V)$\\
  \hline
  10 & $(2c_5-2c_4^2)v_3+(c_4-c_5)v_4+(2c_4-2)v_5$ & $V$\\
  \hline
  11 & $v_5-c_4v_3$ & $V$\\
  \hline
  12 & $(p=2)$ $v_5+c_4v_3$ & $V$\\
  & $(p>2)$ $c_4(c_4-3)v_3+c_4v_4+(1-c_4)v_5$ & $V$\\
  \hline
  13 & $(2c_4^2-2c_5)v_3+(c_5-c_4)v_4+(2-2c_4)v_5$ & $V$\\
  \hline
  14 & $-c_4^2v_3+c_4v_4+(2c_4-2)v_5$ & $V$\\
  \hline
  15 & $v_5$ & $V$\\
  \hline
  16 & $v_1$ & $V$\\
  \hline
  17 & $2v_3+c_5v_4-2v_5$ & $V$\\
  \hline
  18 & $v_3+v_5$ & $V$\\
  \hline
  19 & $(c_6v_3+v_7)\wedge v_4\wedge v_1$ & $\wedge^3(V)$\\
  \hline
    \end{tabular}
\caption{}
  \label{tab:G2_fixed_point}
\end{table}
\FloatBarrier

  For the remaining cases,
  %the argument is based upon determining the precise action
  %of the unipotent group $U_H$ on $V$,
  %finding a fixed point $w$ in some associated representation space $W$
  %%(usually built as an exterior or symmetric power of $V$)
  %which is also fixed by $T_H$ but not fixed by the torus of $G$.
  %For this,
  we indicate in Table~\ref{tab:G2_weights} the weights of $T_H$ on the vectors in $\mathcal B$
  and in Table~\ref{tab:G2_fixed_point} we give a vector which is fixed by $H$ but not by $T$.
  We give the details only for Case 2 and Case 7 to illustrate the most involved cases.
  %\smallbreak

\medbreak
\noindent {\bf{Case 2:}}
Set $q = q_1$. Then $t(\lambda)$ acts on the basis $\mathcal B$ as
$$
{\rm diag}(\lambda^{2qm},\lambda^{qm},\lambda^{qm},\lambda^0,\lambda^{-qm},\lambda^{-qm},\lambda^{-2qm}).
$$
We calculate that $u(x)(v_2-v_3) = v_2-v_3$,
$u(x)v_4 = v_4+2x^q(v_3-v_2)$ and
$u(x)v_6 = v_6-x^qv_4+x^{2q}(v_2-v_3)$.
Now we consider the $G$-module $\wedge^3(V)$,
where $(v_2-v_3)\wedge v_4\wedge v_6$ is fixed by $U_H$ and
has weight $0$ for $T_H$.
Moreover, this vector is not a weight vector for $T$. Hence $T_HU_H$ is not epimorphic in $G$.
\medbreak
\noindent{\bf{Case 7:}} Here the argument is slightly more complicated as we have two different $p$-powers, $q_1$ and $q_5$ and the argument
differs according to their relative sizes. Note first that $t(\lambda)$ acts on the ordered basis $\mathcal B$ as
$${\rm diag}(\lambda^{(q_5-q_1)m},\lambda^{(q_5-2q_1)m},\lambda^{q_1m},\lambda^0,\lambda^{-q_1m},\lambda^{(2q_1-q_5)m},
\lambda^{(q_1-q_5)m}).$$

Consider first the case where $q_1>q_5$.
Here we use the $G$-module $W = S^a(V)\otimes S^b(V)$, where $a = 2q_1-q_5$ and $b = q_1-q_5$.
The vector $\overline{\otimes^{a}v_1}\in S^a(W)$ is fixed by $U_H$ and has $T_H$ weight $a(q_5-q_1)m$, while the vector
$\overline{\otimes^{b}v_6}\in S^b(W)$ is fixed by $U_H$ and has $T_H$ weight $b(2q_1-q_5)$. Thus in $W$, we have the $U_H$-fixed vector $\overline{\otimes^{a}v_1}\otimes\overline{\otimes^{b}v_6}$ of
$T_H$-weight $0$, but which is not fixed by $\alpha_1^{\vee}(k^{\times})\subset T$.

For the case where $q_1=q_5$ it is easy to see that $v_1$ is of weight $0$ for $T_H$ and fixed by $U_H$.
It is also easy to rule out the case where $q_5=2q_1$ since here the vector $v_6$ is of $T_H$ weight $0$ and is fixed by $U_H$.

Finally, we must consider the case where $q_5>2q_1$. Now we set $W = S^a(V)\otimes S^b(V)$, where $a = q_5-2q_1$ and $b = q_5-q_1$, and
we argue as above using $\overline{\otimes^{a}v_1}\otimes\overline{\otimes^{b}v_6}$.
%the $a$-fold tensor $v_1\otimes\cdots\otimes v_1$ and the $t$-fold tensor $v_6\otimes\cdots\otimes v_6$.
\end{proof}

\section{Existence of $3$-dimensional closed epimorphic subgroups}
\label{sec:existence}

In this section, we exhibit a $3$-dimensional epimorphic subgroup $H$ in each of the groups $\SL_3(k)$, $\Sp_4(k)$ and $G_2(k)$, defined over the field $k$, with $\operatorname{char}(k)=p>0$. For $\SL_3(k)$ we may choose $H$ to be the radical of a maximal parabolic subgroup by \S\ref{subsec:criteria}.(8). For the other two groups fix $q\neq 1$ a $p$-power. 

Let $G = {\rm Sp}_4(k)$. Let $A = \langle u_{+}(x), u_{-}(x)\ |\ x\in k\rangle$ be an
$A_1$-type group,
%with root groups $ \langle x(t)\ |\ t\in k\rangle$ and $\langle y(t)\ |\ t\in k\rangle$,
diagonally embedded in the maximal rank subgroup
$\langle U_{\pm\alpha_1}\rangle\times \langle U_{\pm(\alpha_1+2\alpha_2)}\rangle\subseteq {\rm Sp}_4(k)$,
via the morphisms
$$
u_{+}(x)=u_{\alpha_1}(x)u_{\alpha_1+2\alpha_2}(x^{q})
\quad\text{and}\quad
u_{-}(x)=u_{-\alpha_1}(x)u_{-\alpha_1-2\alpha_2}(x^{q}).
$$
Then $A$ has a maximal torus given by $T_A = \{h(\lambda) = \alpha_1^{\vee}(\lambda^{1+q})\alpha_2^{\vee}(\lambda^{2q})\ |\ \lambda\in k^{\times}\}$, so
that $\alpha_1(h(\lambda)) = \lambda^{2}$ and $\alpha_2(h(\lambda)) = \lambda^{q-1}$.

Let $X = \{u_{+}(x)\ |\ x\in k\}$ and consider the subgroup $H = B_AU_{\alpha_1+\alpha_2}$, where $B_A$ is the Borel subgroup $T_{A}X$ of $A$.
The root group $U_{\alpha_1+\alpha_2}$ is normalized by $B_A$, indeed centralised by
  $X$ and $T_A$ acts with weight $q+1$. We claim that $Y=\langle A, U_{\alpha_1+\alpha_2}\rangle = G$, which, by \S\ref{subsec:criteria}.(9), shows that $H$ is epimorphic in $G$.

  Notice that $Y$ acts irreducibly on the $4$-dimensional Weyl module $V$ of highest weight $\omega_2$.
  Indeed, the only $A$-invariant subspaces of $V$ are the non-isomorphic irreducible $A$-modules
  $V_{\omega_2}+V_{\omega_2-\alpha_1-2\alpha_2}$ and $V_{\omega_2-\alpha_2}+V_{\omega_2-\alpha_1-\alpha_2}$.
  Since $U_{\alpha_1+\alpha_2}$ preserves neither of these subspaces,
  $Y$ acts irreducibly on $V$. 
  Now, using \cite[Theorem 3]{Seitz2}, we have the following possibilities for $Y$:%see that one of the following holds:
\begin{itemize}
\item $Y=G$, or
\item $Y$ preserves a nontrivial decomposition $V = V_1\otimes V_2$ and $Y$ is a subgroup of ${\rm Sp}(V_1)\cdot {\rm SO}(V_2)$, or
\item $p=2$ and $Y={\rm SO}(V)$, or
\item $Y$ is simple.
\end{itemize}

Since $\dim Y\geq 4$ and $Y$ has an irreducible $4$-dimensional representation, if $Y$ is simple then $Y=G$. Next we argue that in fact $\dim Y\geq 7$ and thus the second and third  possibilities are ruled
out and we have $Y = G$ as
desired. We have ${\rm Lie}(Y)\supseteq {\rm Lie}(A) + {\rm Lie}(U_{\alpha_1+\alpha_2})$. Let $A$ act on the root vector
$e_{\alpha_1+\alpha_2}$, which is a maximal vector of an $A$-composition factor of ${\rm Lie}(G)$. The weight of this vector being $q+1$, we see that there is a $4$-dimensional irreducible composition factor,
and we now have 7 distinct weights in the action of $A$ on ${\rm Lie}(Y)$, establishing the claim.

%Note : probably I should let $q_1=1$ so that the above embedding has nonzero differential as well. But I do not think this will change the argument. \bigbreak

Now turn to the case of $G=G_2(k)$. Here again, we set $A =\langle u_{+}(x), u_{-}(x)\ |\ x\in k\rangle$ to be an $A_1$-subgroup
diagonally embedded in the maximal rank subgroup $\langle U_{\pm\alpha_1}\rangle\times \langle U_{\pm(3\alpha_1+2\alpha_2)}\rangle\subseteq G_2(k)$,
via the morphisms
$$
u_{+}(x)=u_{\alpha_1}(x)u_{3\alpha_1+2\alpha_2}(x^{q})
\quad\text{and}\quad
u_{-}(x)=u_{-\alpha_1}(x)u_{-3\alpha_1-2\alpha_2}(x^{q}).
$$
Arguing as above, we see that $B_A$ normalises the root group $U_{3\alpha_1+\alpha_2}$ and we set $H = B_AU_{3\alpha_1+\alpha_2}$ and
$Y = \langle A,U_{3\alpha_1+\alpha_2}\rangle$. We will show that $Y$ acts irreducibly on the Weyl module $V$ of highest weight $\omega_1$.
Then noting that $\dim Y\geq 4$, the main theorem of \cite{Testerman1} shows that either $Y=G$, or $Y$ is non-simple semisimple, or  $p=3$ and $Y$ lies in a short root subsystem subgroup of type
$A_2$. The latter is not possible as $Y\supset U_{3\alpha_1+\alpha_2}$. The second case is also not possible: there is no non-simple semisimple algebraic group with a $7$-dimensional faithful irreducible
representation and the only such group with a $6$-dimensional faithful irreducible representation when $p=2$ is of type $A_2A_1$ which cannot appear as a subgroup of $G_2(k)$. Thus we conclude that $Y=G$ and we complete the argument by applying \S\ref{subsec:criteria}.(9).

To see that $Y$ acts irreducibly, we first consider the action of $A$ on $V$. Here we have two non-isomorphic
summands $V_{\omega_1}+V_{\omega_1-\alpha_1}+V_{\omega_1-3\alpha_1-2\alpha_2}+V_{\omega_1-4\alpha_1-2\alpha_2}$ and
$V_{\omega_1-\alpha_1-\alpha_2}+V_{\omega_1-2\alpha_1-\alpha_2}+V_{\omega_1-3\alpha_1-\alpha_2}$. (Note that if $p=2$, we have
$V_{\omega_1-2\alpha_1-\alpha_2}=0$.) Moreover, $U_{3\alpha_1+\alpha_2}$ stabilizes neither of the two subspaces.
%(this can be checked with the
%explicit matrix representation given in the proof of Proposition \ref{prop:non-existence}.)
Hence $Y$ acts irreducibly on $V$ as claimed.

\end{document}